\documentclass[reqno]{amsart}
\usepackage[backref=page]{hyperref}
\usepackage[nameinlink,capitalize]{cleveref}
\usepackage{cite}
\usepackage{enumitem}

\theoremstyle{plain}
\newtheorem{theorem}{Theorem}

\theoremstyle{definition}

\theoremstyle{remark}
\newtheorem{remark}{Remark}

\allowdisplaybreaks

\begin{document}
	
	\author{Phuong Le}
	\address{Phuong Le$^{1,2}$ (ORCID: 0000-0003-4724-7118)\newline
		$^1$Faculty of Economic Mathematics, University of Economics and Law, Ho Chi Minh City, Vietnam; \newline
		$^2$Vietnam National University, Ho Chi Minh City, Vietnam}
	\email{phuongl@uel.edu.vn}
	
	\subjclass[2020]{35J92, 35B06, 35A02, 35B09}
	\keywords{$p$-Laplace equation, critical Hardy-Sobolev exponent, Hardy potential, classification of solutions}
	
	
	\title[On $p$-Laplace equations with a critical Hardy-Sobolev exponent]{On $p$-Laplace equations with a critical Hardy-Sobolev exponent and a Hardy potential}
	\begin{abstract}
		For $N\ge2$ and $1<p<N$, we classify all positive $\mathcal{D}^{1,p}(\mathbb{R}^N)$-solutions to $p$-Laplace equations with a critical Hardy-Sobolev exponent and a Hardy potential.
	\end{abstract}
	
	\maketitle
	
	\section{Introduction}
	
	This paper is devoted to the doubly critical problem
	\begin{equation}\label{main}
		\begin{cases}
			-\Delta_p u - \dfrac{\mu}{|x|^p} u^{p-1} = \dfrac{u^{p^*_s-1}}{|x|^s} & \text{ in } \mathbb{R}^N,\\
			u>0 & \text{ in } \mathbb{R}^N,\\
			u \in \mathcal{D}^{1,p}(\mathbb{R}^N),
		\end{cases}
	\end{equation}
	where $N\ge2$, $1<p<N$, $0<\mu<\left(\frac{N-p}{p}\right)^p$, $0< s<p$, $p^*_s:=\frac{(N-s)p}{N-p}$ is the critical Hardy-Sobolev exponent, $\Delta_p u={\rm div}(|\nabla u|^{p-2}\nabla u)$ is the $p$-Laplacian of $u$ and
	\[
	\mathcal{D}^{1,p}(\mathbb{R}^N) = \left\{ u \in L^{p^*}(\mathbb{R}^N) \mid \int_{\mathbb{R}^N} |\nabla u|^p dx < \infty \right\}
	\]
	with $p^*:=\frac{Np}{N-p}$. Solutions $u$ to \eqref{main} is understood in the following weak sense
	\begin{equation}\label{weak_sol}
		\int_{\mathbb{R}^N} |\nabla u|^{p-2} \nabla u \cdot \nabla \varphi - \mu \int_{\mathbb{R}^N} \frac{1}{|x|^p} u^{p-1} \varphi = \int_{\mathbb{R}^N} \frac{u^{p^*_s-1}}{|x|^s} \varphi \quad\text{ for all } \varphi \in C^1_c(\mathbb{R}^N\setminus\{0\}).
	\end{equation}
	
	By standard elliptic estimates, we know that any solution $u$ of \eqref{main} belongs to $C^{1,\alpha}_{\rm loc}(\mathbb{R}^N\setminus\{0\})$ for some $0<\alpha<1$ (see \cite{MR709038,MR969499,MR727034}). Due to the assumption $u \in \mathcal{D}^{1,p}(\mathbb{R}^N)$, solutions to \eqref{main} are usually called \textit{finite energy} ones. It is worth pointing out that problem \eqref{main} naturally arises as the Euler-Lagrange equation of the following Hardy–Sobolev–Maz'ya inequality
	\[
	\int_{\mathbb{R}^N} \left(|\nabla u|^p - \frac{\mu}{|x|^p} |u|^p\right) dx \ge C \left(\int_{\mathbb{R}^N} \frac{|u|^{p^*_s}}{|x|^s} dx\right)^\frac{p}{p^*_s},
	\]
	which holds for all $u \in \mathcal{D}^{1,p}(\mathbb{R}^N)$ and some optimal constant $C>0$. This inequality is a consequence of the well known Hardy inequality
	\[
	\int_{\mathbb{R}^N} |\nabla u|^p dx \ge \left(\frac{N-p}{p}\right)^p \int_{\mathbb{R}^N} \frac{1}{|x|^p} |u|^p\ dx
	\]
	and the Sobolev–Hardy inequality
	\[
	\int_{\mathbb{R}^N} |\nabla u|^p dx \ge C' \left(\int_{\mathbb{R}^N} \frac{|u|^{p^*_s}}{|x|^s} dx\right)^\frac{p}{p^*_s},
	\]
	which both hold for all $u \in \mathcal{D}^{1,p}(\mathbb{R}^N)$ and some constant $C'>0$ (see \cite{MR768824,MR1616905}).
	
	Problem \eqref{main} with $p=2$ was studied extensively in the literature due to the availability of several analytic tools for the Laplace operator. In particular, Caffarelli, Gidas and Spruck \cite{MR982351} proved that every classical positive solution to the problem $-\Delta u = u^\frac{N+2}{N-2}$ in $\mathbb{R}^N$ ($N\ge3$) is radial and hence classified by the formula
	\[
	u(x) = \left(\frac{\sqrt{N(N-2)} \lambda}{\lambda^2 + |x-x_0|^2}\right)^\frac{N-2}{2}
	\]
	for some $\lambda>0$ and $x_0\in\mathbb{R}^N$.
	Later, Chen and Li \cite{MR1121147} simplified the proof by combining the moving plane method \cite{MR143162,MR333220} with the Kelvin transform. The same method can be applied to to the problem $-\Delta u = \frac{u^{2^*_s-1}}{|x|^s}$ in $\mathbb{R}^N$, where $N\ge3$, $0<s<2$, to yield the classification of all positive solutions in $H^1_{\rm loc}(\mathbb{R}^N)\cap L^\infty_{\rm loc}(\mathbb{R}^N)$ (see \cite{2017arXiv170304353G} for instance). We stress that the semilinear critical problem is invariant under the Kelvin transform and the transform provides the right decaying properties of solutions so that the moving plane method on the whole space can be carried out. Hence the above results hold for all positive solutions without the finite energy assumption $u \in \mathcal{D}^{1,2}(\mathbb{R}^N)$. We also mention that all the solutions of the problem \eqref{main} for $p =2$ and $s=0$ have been classified by Terracini \cite{MR1364003}.
	
	Next, we consider the quasilinear case $p\ne2$. This case is much more difficult due to the nonlinear nature of the $p$-Laplacian, the lack of regularity of the solutions and the fact that comparison principles are not equivalent to maximum principles in this case. Furthermore, the Kelvin type transform is not available for $p$-Laplace equations with $p\ne2$. To overcome this difficulty, the main approach is to establish optimal estimates on asymptotic behaviors of solutions, then exploit the moving plane method or an integral identity to get the symmetry or the explicit form of the solutions. We summarize some main achievements regarding problem \eqref{main} in the quasilinear case below.
	\begin{itemize}
		\item Solutions to problem \eqref{main} with $\mu=s=0$ was classified by Sciunzi \cite{MR3459013} and V\'{e}tois \cite{MR3411668} via the moving plane method (the radial solutions were studied before in \cite{MR964617} and the symmetry of minimizers of the corresponding Sobolev inequality were derived in \cite{MR463908}). Later, Ciraolo, Figalli and Roncoroni \cite{MR4135671} introduced a new approach exploiting an integral formula to classify $\mathcal{D}^{1,p}(\Sigma)$ positive solutions to critical anisotropic $p$-Laplace equations in any convex cones $\Sigma\subset\mathbb{R}^N$.
		\item The case $\mu>s=0$ was studied by Oliva, Sciunzi and Vaira \cite{MR4124427} via the moving plane method exploiting some estimates in \cite{MR3369267,MR3581526} (see also \cite{MR2233146} for the uniqueness of radial solutions and the symmetry of minimizers).
		\item Solutions to \eqref{main} in the case $s>\mu=0$ have been classified recently in \cite{MR4611678,MR4343877} exploiting estimates in \cite{MR4194414} (see also \cite{MR1695021} for the radial case).
		\item In the general case that $\mu>0$, $s>0$, the uniqueness and asymptotic estimates for radial solutions are available, see \cite[Lemma 2.3]{MR2388757} and \cite[Theorem 1.4]{MR3467702} (see also \cite{MR1695021}). Some estimates for nonradial solutions have been obtained recently in \cite{MR4628991}.
	\end{itemize}
	
	In this paper, we follow the approach in \cite{MR4124427} to prove the symmetry and hence the uniqueness of solutions to \eqref{main}. Our main result is the following theorem.
	\begin{theorem}\label{th:main}
		Assume $1<p<N$, $0<\mu<\left(\frac{N-p}{p}\right)^p$, $0< s<p$. Let $u$ be a solution to \eqref{main}. Then $u$ is radial and radially decreasing with respect to the origin.
	\end{theorem}
	\begin{remark}
		Combining Theorem \ref{th:main} with Theorem 1.4 in \cite{MR3467702}, we deduce that problem \eqref{main} has a unique solution up to a dilation. This means that if $u$ is a solution to \eqref{main}, then all of its solutions are of the form $u_\tau(x) := \tau^\frac{N-p}{p} u(\tau x)$ for some $\tau>0$.
	\end{remark}
	
	The remainder of this paper is divided into two sections. In Section \ref{sect2}, we recall some known results and prove some sharp pointwise gradient estimates which will be used later. Section \ref{sect3} is devoted to the proof of our main result, namely, Theorem \ref{th:main}, via the method of moving planes.
	
	\section{Preliminaries and asymptotic estimates}\label{sect2}
	
	It is important to study the summability of the second derivatives of the solutions to $p$-Laplace equations since such solutions are generally not in the $C^2$ class. Such results, which are due to Damascelli and Sciunzi \cite{MR2096703}, will be recalled here.
	
	\begin{theorem}[Hessian and reversed gradient estimates \cite{MR2096703}]\label{th:hessian}
		Let $u$ be a solution to \eqref{main}. Then for $i=1,\dots,N$ we have
		\[
		\int_{\Omega} \frac{|\nabla u|^{p-2-\beta} |\nabla \frac{\partial u}{\partial x_i}|^2}{|x-y|^\gamma} dx \le C
		\]
		for any $\Omega \subset\subset \mathbb{R}^N\setminus\{0\}$ and uniformly for any $y \in \Omega$, where $0\le\beta<1$ and $\gamma<N-2$ if $N\ge3$, $\gamma=0$ if $N=2$ and
		\[
		C = C(p,\beta,\gamma,\mu,s,\|u\|_{L^\infty(\Omega)},\|\nabla u\|_{L^\infty(\Omega)},{\rm dist}(\Omega,\{0\})).
		\]
		Moreover,
		\begin{equation}\label{inverse_int}
			\int_{\Omega} \frac{1}{|\nabla u|^t} \frac{1}{|x-y|^\gamma} dx \le C^* \text{ uniformly for any } y \in \Omega,
		\end{equation}
		where $(p-2)^+ \le t < p-1$ and $\gamma<N-2$ if $N\ge3$, $\gamma=0$ if $N=2$ and $C^*$ depends on $C$.
	\end{theorem}
	The proof of Theorem \ref{th:hessian} is almost the same as that of \cite[Theorem 1.1]{MR2096703} with some minor changes and hence will be omitted.
	
	Let $\rho\in L^1(\Omega)$ be a positive function, where $\Omega$ is a bounded domain of $\mathbb{R}^N$. We define $H^1(\Omega,\rho)$ as the completion of $C^1(\Omega)$ with the norm
	\[
	\|v\|_{H^1(\Omega,\rho)} = \|v\|_{L^2(\Omega)} + \|\nabla v\|_{L^2(\Omega,\rho)},
	\]
	where $\|\nabla v\|_{L^2(\Omega,\rho)}^2 = \int_{\Omega} \rho |\nabla v|^2 dx$. We also define $H^1_0(\Omega,\rho)$ as the closure of $C^\infty_c(\Omega)$ in $H^1(\Omega,\rho)$.
	The following weighted Poincar\'{e} type inequality will be utilized later.
	\begin{theorem}[Weighted Poincar\'{e} type inequality \cite{MR2096703}]\label{th:wp_ine}
		Let $\rho$ be a positive function such that
		\begin{equation}\label{wp_ine_cond1}
			\int_\Omega \frac{1}{\rho |x-y|^\gamma} dy \le C_1 \quad\text{ for all } x \in \Omega,
		\end{equation}
		where $\Omega$ is a bounded domain in $\mathbb{R}^N$ and $\gamma<N-2$ if $N\ge3$, $\gamma=0$ if $N=2$.
		Let $w \in H^1(\Omega,\rho)$ be such that
		\begin{equation}\label{wp_ine_cond2}
			|w(x)| \le C_2 \int_\Omega \frac{|\nabla w(y)|}{|x-y|^{N-1}} dy \quad\text{ for all } x \in \Omega.
		\end{equation}
		Then we have
		\[
		\int_\Omega w^2 \le C_P \int_\Omega \rho |\nabla w|^2,
		\]
		where $C_P$ depends on $\Omega$ and $C_1$. Furthermore, $C_P\to0$ as $|\Omega|\to0$. The same inequality holds for $w \in H^1_0(\Omega,\rho)$.
	\end{theorem}
	Here and throughout the paper, we use $|\Omega|$ to denote the Lebesgue measure of a measurable set $\Omega\subset\mathbb{R}^N$. We will also use $C$ and $c$ to denote generic positive constants which may change from line to line or even in the same line. Theorem \ref{th:wp_ine} follows from the proof of \cite[Theorem 3.1]{MR2096703}, where condition \eqref{wp_ine_cond2} is used to obtain the key estimate (3.9) in \cite{MR2096703}.
	
	Throughout the paper, we denote by $\gamma_1<\gamma_2$ two roots of the equation
	\[
	(p-1)\gamma^p - (N-p)\gamma^{p-1} + \mu = 0.
	\]
	Then
	\[
	0 < \gamma_1 < \frac{N-p}{p} < \gamma_2 < \frac{N-p}{p-1}.
	\]
	We recall the following asymptotic estimates for \eqref{main}, which were obtained recently in \cite{MR4628991} (see also \cite{MR3369267}). 
	\begin{theorem}[Asymptotic estimates \cite{MR3369267,MR4628991}]\label{th:asymptotic}
		Let $u$ be a solution to \eqref{main}. Then
		\begin{align*}
			c |x|^{-\gamma_1} \le u(x) \le C |x|^{-\gamma_1} &\quad \text{ for } |x|<R_0,\\
			c |x|^{-\gamma_2} \le u(x) \le C |x|^{-\gamma_2} &\quad \text{ for } |x|>R_1,
		\end{align*}
		where $C,c>0$ and $0<R_0<1<R_1$ are constants depending on $N,p,\mu, s$ and the solution $u$.
	\end{theorem}
	
	By continuity, it is clear that for any given $R_0,R_1>0$, the estimates in Theorem \ref{th:asymptotic} still hold for suitable choice of $C,c>0$. By exploiting Theorem \ref{th:asymptotic}, we will prove the following asymptotic estimates of the gradient of solutions.
	\begin{theorem}[Asymptotic gradient estimates]\label{th:gradient_asymptotic}
		Let $u$ be a solution to \eqref{main}. Then
		\begin{align*}
			c |x|^{-(\gamma_1+1)} \le |\nabla u(x)| \le C |x|^{-(\gamma_1+1)} &\quad \text{ for } |x|<R_0',\\
			c |x|^{-(\gamma_2+1)} \le |\nabla u(x)| \le C |x|^{-(\gamma_2+1)} &\quad \text{ for } |x|>R_1',
		\end{align*}
		where $C,c>0$ and $0<R_0'<1<R_1'$ are constants depending on $N,p,\mu,s$ and the solution $u$.
	\end{theorem}
	
	Theorems \ref{th:asymptotic} and \ref{th:gradient_asymptotic} show that solutions to \eqref{main} with $s>0$ exhibit the same asymptotic behaviors as in the case $s=0$. The proof of Theorem \ref{th:gradient_asymptotic} relies on the following classification result.
	\begin{theorem}[Theorem 3.2 in \cite{MR4124427}]\label{th:classification}
		Let $1 <p < N$, $0<\mu<\left(\frac{N-p}{p}\right)^p$ and let $v\in C^{1,\alpha}_{\rm loc}(\mathbb{R}^N\setminus\{0\})$ with $0<\alpha<1$ be a positive solution to
		\[
		-\Delta_p v - \frac{\mu}{|x|^p} v^{p-1} = 0 \text{ in } \mathbb{R}^N \setminus\{0\},
		\]
		such that
		\[
		\lim_{|x|\to0} v(x) = +\infty, \quad \lim_{|x|\to+\infty} v(x) = 0.
		\]
		Then $v$ is a radially decreasing function (i.e., $v(x)=v(|x|)$ with $v'<0$).
	\end{theorem}
	
	\begin{proof}[Proof of Theorem \ref{th:gradient_asymptotic}]
		We will follow the scaling technique in \cite{MR3459013} to show that
		\begin{equation}\label{grad_est}
			c |x|^{-(\gamma_2+1)} \le |\nabla u(x)| \le C |x|^{-(\gamma_2+1)} \quad \text{ for } |x|>R_1',
		\end{equation}		
		where $C,c>0$ and $R_1'>0$ are constants depending on $N,p,\mu,s$ and the solution $u$. The estimates for $|x|<R_0'$ are similar and can be proved similarly.
		
		For $a,A$ such that $0 < a < R_0 < R_1 < A$ and $R_n\to+\infty$, let us consider
		\[
		w_n(x) := R_n^{\gamma_2} u(R_n x) \quad\text{ for } x \in \mathbb{R}^N\setminus\{0\}.
		\]
		For $n$ sufficiently large, from Theorem \ref{th:asymptotic}, we have
		\[
		\frac{c}{A^{\gamma_2}} \le w_n(x) \le \frac{C}{a^{\gamma_2}} \quad\text{ in } B_A(0)\setminus B_a(0)
		\]
		and
		\begin{equation}\label{wn_blowup}
			w_n(x) \le \frac{C}{A^{\gamma_2}} \text{ for } x\in \partial B_A(0), \quad
			w_n(x) \ge \frac{c}{a^{\gamma_2}} \text{ for } x\in \partial B_a(0).
		\end{equation}
		In particular, $(w_n)$ is uniformly bounded in $L^\infty(B_A(0)\setminus B_a(0))$ and it weakly solves
		\begin{equation}\label{wn}
			-\Delta_p w_n - \frac{\mu}{|x|^p} w_n^{p-1} = \frac{1}{R_n^{(p_s^*-p)\gamma_2+s-p}} \frac{w_n^{p^*_s-1}}{|x|^s} \quad \text{ in } \mathbb{R}^N.
		\end{equation}
		Notice that $(p_s^*-p)\gamma_2+s-p>0$ thanks to $\gamma_2>\frac{N-p}{p}$. By standard regularity results in \cite{MR709038,MR969499,MR727034}, $(w_n)$ is also uniformly bounded in $C^{1,\alpha}(K)$, for $0 < \alpha < 1$ and for any compact set $K \subset B_A(0)\setminus B_a(0)$. Since $\nabla w_n(x) = R_n^{\gamma_2+1} \nabla u(R_n x)$, for $R_n$ sufficiently large we get the
		estimate from above in \eqref{grad_est}.
		
		Now we prove the estimate from below. Suppose by contradiction that there exists a sequence of points $x_n$ such that
		\begin{equation}\label{cassump}
			|x_n|^{\gamma_2+1} |\nabla u(x_n)| \to 0 \quad\text{ for } |x_n| \to +\infty.
		\end{equation}
		Since $(w_n)$ is uniformly bounded in $C^{1,\alpha}(K)$, up to a subsequence, we have
		\[
		w_n \to w_{a,A} \quad\text{ in } C^{1,\alpha'}(B_A(0)\setminus B_a(0))
		\]
		for $0<\alpha'<\alpha$. Moreover, passing \eqref{wn} to the limit, we get
		\[
		-\Delta_p w_{a,A} - \frac{\mu}{|x|^p} w_{a,A}^{p-1} = 0 \quad\text{ in } C^{1,\alpha'}(B_A(0)\setminus B_a(0)).
		\]
		Now we take $a = \frac{1}{j}$ and $A = j$, for large $j \in \mathbb{N}$ and we construct $w_{\frac{1}{j},j}$ as above. For $j\to\infty$, using a standard diagonal process, we can construct a limiting profile $w_\infty$ so that
		\[
		-\Delta_p w_\infty - \frac{\mu}{|x|^p} w_\infty^{p-1} = 0 \quad\text{ in } \mathbb{R}^N \setminus\{0\}
		\]
		and $w_{\frac{1}{j},j} = w_\infty$ in $B_j(0)\setminus B_{\frac{1}{j}}(0)$. Moreover, from \eqref{wn_blowup} we know that
		\[
		\lim_{|x|\to0} w_\infty(x) = +\infty, \quad \lim_{|x|\to+\infty} w_\infty(x) = 0.
		\]
		By Theorem \ref{th:classification}, $w_\infty$ is a radially decreasing function with $w_\infty'<0$.
		
		Now we set $R_n=|x_n|$ and $y_n = \frac{x_n}{R_n}$, where $x_n$ is determined in \eqref{cassump}. Then we deduce that $|\nabla w_n(y_n)| = |x_n|^{\gamma_2+1} |\nabla u(x_n)| \to 0$ as $n \to \infty$. This fact and the uniform convergence of the gradients imply that there exists $\overline{y} \in \partial B_1(0)$ such that
		\[
		|\nabla w_\infty(\overline{y})| = 0.
		\]
		This is a contradiction since the solution $w_\infty$ has no critical points.
	\end{proof}
	
	\section{Radial symmetry of solutions} \label{sect3}
	Our proof is based on the moving plane method. For each $\lambda\le 0$, we denote
	\[
	\Sigma_\lambda := \{ x \in \mathbb{R}^N \mid x_1 < \lambda \}.
	\]
	We also denote by $x_\lambda$ the reflection of $x$ with respect to $\Sigma_\lambda$, i.e.,
	\[
	x_\lambda := (2\lambda - x_1, x') \quad\text{ for } x =(x_1, x') \in\mathbb{R}\times\mathbb{R}^{N-1}.
	\]
	Furthermore, we set
	\[
	u_\lambda(x) := u(x_\lambda).
	\]
	Then $u_\lambda$ weakly solves
	\begin{equation}\label{main2}
		-\Delta_p u_\lambda - \frac{\mu}{|x_\lambda|^p} u_\lambda^{p-1} = \frac{u_\lambda^{p^*_s-1}}{|x_\lambda|^s}.
	\end{equation}
	
	\begin{proof}[Proof of Theorem \ref{th:main}]
		We mainly follow the proof of Theorem 1.1 in \cite{MR4124427}. Let
		\[
		\Lambda = \{ \lambda<0 \mid u \le u_\gamma \text{ in } \Sigma_\gamma \text{ for all } \gamma \le \lambda \}.
		\]
		
		\textit{Step 1.} We claim that $\Lambda\neq\emptyset$.
		
		By Theorem \ref{th:asymptotic}, we have that $\lim_{|x|\to0} u(x) = +\infty$ and $\lim_{|x|\to+\infty} u(x) = 0$. Hence there exists $\overline{R}_0>0$ such that
		\[
		\sup_{x\in B_{\overline{R}_0}(0_\lambda)} u(x) < \inf_{x\in B_{\overline{R}_0}(0)} u(x)
		\]
		for all $\lambda<-R_1$. This implies
		\begin{equation}\label{ball_ineq}
			u<u_\lambda \text{ in } B_{\overline{R}_0}(0_\lambda)\quad \text{ for all } \lambda<-R_1.
		\end{equation}
		
		Now we take $R>0$ and denote by $\eta \in C^\infty_0(B_{2R}(0))$ a cut-off function such that $0\le\eta\le1$, $\eta\equiv1$ on $B_R(0)$ and $|\nabla\eta|\le\frac{2}{R}$. In what follows we employ the following notation $\Sigma_\lambda' = \Sigma_\lambda \setminus B_{\overline{R}_0}(0_\lambda)$ and $\hat B_R = B_R(0) \cap \Sigma_\lambda'$. For $\alpha>\max\{2, p\}$ and $\lambda<-R_1$, we consider
		\[
		\varphi_{1,\lambda} := \eta^\alpha u^{1-p} (u^p - u_\lambda^p)^+ \chi_{\Sigma_\lambda},\quad
		\varphi_{2,\lambda} := \eta^\alpha u_\lambda^{1-p} (u^p - u_\lambda^p)^+ \chi_{\Sigma_\lambda}.
		\]
		We also denote
		\[
		\psi_\lambda := (u^p - u_\lambda^p)^+,\quad
		\varphi_\lambda := (u - u_\lambda)^+.
		\]
		
		Using $\varphi_{1,\lambda}$ as a test function in \eqref{weak_sol} and $\varphi_{2,\lambda}$ as a test function in the distributional formulation of \eqref{main2} and subtracting, we obtain
		\begin{align*}
			&\int_{\hat B_{2R}} (|\nabla u|^{p-2} \nabla u \cdot \nabla \varphi_{1,\lambda} - |\nabla u_\lambda|^{p-2} \nabla u_\lambda \cdot \nabla \varphi_{2,\lambda}) + \mu \int_{\hat B_{2R}} \left(-\frac{1}{|x|^p} + \frac{1}{|x_\lambda|^p}\right) \eta^\alpha \psi_\lambda\\
			&= \int_{\hat B_{2R}} \left(\frac{u^{p^*_s-p}}{|x|^s} - \frac{u_\lambda^{p^*_s-p}}{|x_\lambda|^s}\right) \eta^\alpha \psi_\lambda.
		\end{align*}
		Since $|x| > |x_\lambda|$ in $\Sigma_\lambda$, one has that the second term on the left hand side is nonnegative. Hence
		\begin{equation}\label{I}
			\begin{aligned}
				&\underbrace{\int_{\hat B_{2R}} \eta^\alpha (|\nabla u|^{p-2} \nabla u \cdot \nabla (u^{1-p}\psi_\lambda) - |\nabla u_\lambda|^{p-2} \nabla u_\lambda \cdot \nabla (u_\lambda^{1-p}\psi_\lambda))}_{I_1}\\
				&\le \underbrace{-\alpha\int_{\hat B_{2R}} \eta^{\alpha-1} u^{1-p} \psi_\lambda |\nabla u|^{p-2} \nabla u \cdot \nabla \eta}_{I_2} + \underbrace{\alpha\int_{\hat B_{2R}} \eta^{\alpha-1} u_\lambda^{1-p} \psi_\lambda |\nabla u_\lambda|^{p-2} \nabla u_\lambda \cdot \nabla \eta}_{I_3}\\
				&\qquad+ \underbrace{\int_{\hat B_{2R}} \left(\frac{u^{p^*_s-p}}{|x|^s} - \frac{u_\lambda^{p^*_s-p}}{|x_\lambda|^s}\right) \eta^\alpha \psi_\lambda}_{I_4}.
			\end{aligned}
		\end{equation}
		
		Working as in the proof of Theorem 1.1 in \cite{MR4124427}, we have the following estimate for $I_1$:
		\begin{equation}\label{I1.1}
			I_1 \ge c_1 \int_{\hat B_{2R}\cap\{u\ge u_\lambda\}} \eta^\alpha u^2 (|\nabla u| + |\nabla u_\lambda|)^{p-2} |\nabla\ln u - \nabla\ln u_\lambda|^2
		\end{equation}
		for $p>2$, and
		\begin{equation}\label{I1.2}
			I_1 \ge c_1 \int_{\hat B_{2R}\cap\{u\ge u_\lambda\}} \eta^\alpha u_\lambda^2 (|\nabla u| + |\nabla u_\lambda|)^{p-2} |\nabla\ln u - \nabla\ln u_\lambda|^2
		\end{equation}
		for $1<p<2$. We also have the following estimates for $I_2$ and $I_3$
		\begin{equation}\label{I2-3}
			I_2 \le \frac{C}{R^\beta},\quad
			I_3 \le \frac{C}{R^\frac{\beta}{p}},
		\end{equation}
		where $\beta := p\gamma_2 + p - N$, which is positive since $\gamma_2>\frac{N-p}{p}$.
		Now we estimate $I_4$ as follows
		\begin{align*}
			I_4 &\le \int_{\hat B_{2R}} \left(\frac{u^{p^*_s-p}}{|x|^s} - \frac{u_\lambda^{p^*_s-p}}{|x|^s}\right) \eta^\alpha \psi_\lambda \le \int_{\hat B_{2R}} \frac{1}{|x|^s u^{p-1}} \left(u^{p^*_s-1} - u_\lambda^{p^*_s-1}\right) \eta^\alpha \psi_\lambda.
		\end{align*}
		Applying twice the inequality
		\[
		b^q - a^q \le \max\{q,1\} b^{q-1} (b-a) \quad\text{ for } 0<a<b \text{ and } q\ge0
		\]
		and using Theorem \ref{th:asymptotic} one has that
		\begin{equation}\label{I4'}
			I_4 \le \max\{p^*_s-1,1\} \max\{p,1\} \int_{\hat B_{2R}} \frac{u^{p^*_s-2}}{|x|^s} \eta^\alpha \varphi_\lambda^2 \le C \int_{\hat B_{2R}} \frac{1}{|x|^{s+\gamma_2 (p^*_s-2)}} \eta^\alpha \varphi_\lambda^2.
		\end{equation}
		
		For any $0<a<b$, by the Lagrange theorem for the function $f(t)=\ln t$ we have
		\begin{equation}\label{ln_ineq}
			b - a \le b(\ln b - \ln a).
		\end{equation}
		
		We use \eqref{ln_ineq} with $b=u$ and $a=u_\lambda$ and estimate the right hand side of \eqref{I4'} as
		\begin{align*}
			I_4 &\le C \int_{\hat B_{2R}\cap\{u\ge u_\lambda\}} \frac{1}{|x|^{s+\gamma_2 (p^*_s-2)}} \eta^\alpha u^2 (\ln u - \ln u_\lambda)^2\\
			&\le C \int_{\hat B_{2R}} \frac{1}{|x|^{s+\gamma_2 p^*_s}} \eta^\alpha \left((\ln u - \ln u_\lambda)^+\right)^2\\
			&= C \int_{\hat B_{2R}} \frac{1}{|x|^{\beta^*-2 \rho + 2}} \eta^\alpha \left((\ln u - \ln u_\lambda)^+\right)^2\\
			&\le \frac{C}{|\lambda|^{\beta^*}} \int_{\hat B_{2R}} |x|^{2 \rho - 2} \left(\eta^\frac{\alpha}{2}(\ln u - \ln u_\lambda)^+\right)^2,
		\end{align*}
		where $\beta^* = \gamma_2 (p_s^*-p)+s-p$ and $2\rho = -[(\gamma_2+1)(p-2) + 2\gamma_2]$. Notice that $\beta^*-2 \rho + 2 = s+\gamma_2 p^*_s$ and $\beta^*>0$ since $\gamma_2>\frac{N-p}{p}$.
		
		Now we can proceed as in the proof of Theorem 1.1 in \cite{MR4124427} to obtain
		\begin{equation}\label{I4.1}
			I_4 \le \frac{C}{|\lambda|^{\beta^*}} \int_{\hat B_{2R}\cap\{u\ge u_\lambda\}} \eta^\alpha u^2 (|\nabla u| + |\nabla u_\lambda|)^{p-2} |\nabla\ln u - \nabla\ln u_\lambda|^2 + \frac{C}{|\lambda|^{\beta^*} R^\beta}
		\end{equation}
		for $p>2$, and
		\begin{equation}\label{I4.2}
			I_4 \le \frac{C}{|\lambda|^{\beta^*}} \int_{\hat B_{2R}\cap\{u\ge u_\lambda\}} \eta^\alpha u_\lambda^2 (|\nabla u| + |\nabla u_\lambda|)^{p-2} |\nabla\ln u - \nabla\ln u_\lambda|^2 + \frac{C}{|\lambda|^{\beta^*} R^\beta}
		\end{equation}
		for $1<p<2$.
		
		Hence, by collecting \eqref{I1.1}, \eqref{I2-3}, \eqref{I4.1} when $p>2$ and \eqref{I1.2}, \eqref{I2-3}, \eqref{I4.2} when $1<p<2$ in \eqref{I}, we deduce
		\begin{align*}
			&\left(c_1 - \frac{C}{|\lambda|^{\beta^*}}\right) \int_{\hat B_{2R}\cap\{u\ge u_\lambda\}} \eta^\alpha u_\lambda^2 (|\nabla u| + |\nabla u_\lambda|)^{p-2} |\nabla\ln u - \nabla\ln u_\lambda|^2\\
			&\le \frac{C}{R^\beta} + \frac{C}{R^\frac{\beta}{p}} +\frac{C}{|\lambda|^{\beta^*} R^\beta}.
		\end{align*}
		
		We can choose $|\lambda|$ large enough so that, as $R\to+\infty$, it yields
		\[
		\int_{\Sigma_\lambda'\cap\{u\ge u_\lambda\}} u_\lambda^2 (|\nabla u| + |\nabla u_\lambda|)^{p-2} |\nabla\ln u - \nabla\ln u_\lambda|^2 = 0.
		\]
		Hence $\ln u - \ln u_\lambda$ is constant and since $\ln u - \ln u_\lambda = 0$ on $T_\lambda$ then $\ln u - \ln u_\lambda = 0$ on the set $\Sigma_\lambda'\cap\{u\ge u_\lambda\}$. From this and \eqref{ball_ineq} we get $u\le u_\lambda$ in $\Sigma_\lambda$. Hence $\Lambda\neq\emptyset$ and we can set
		\[
		\lambda_0 = \sup \Lambda.
		\]
		Then $\lambda_0 \le 0$.
		
		\textit{Step 2.} We claim that $\lambda_0=0$.
		
		Assume by contradiction that $\lambda_0 < 0$. By continuity, we have $u \le u_{\lambda_0}$ in $\Sigma_{\lambda_0}$. Since $u, u_{\lambda_0} \in C^{1,\alpha}_{\rm loc}(\Sigma_{\lambda_0}\setminus\{0_{\lambda_0}\})$ and
		\begin{align*}
			-\Delta_p u - \frac{\mu}{|x|^p} u^{p-1} - \frac{u^{{p^*_s}-1}}{|x|^s} &=0= -\Delta_p u_{\lambda_0} - \frac{\mu}{|x_{\lambda_0}|^p} u_{\lambda_0}^{p-1} - \frac{u_{\lambda_0}^{{p^*_s}-1}}{|x_{\lambda_0}|^s}\\
			&\le -\Delta_p u_{\lambda_0} - \frac{\mu}{|x|^p} u_{\lambda_0}^{p-1} - \frac{u_{\lambda_0}^{{p^*_s}-1}}{|x|^s} \quad\text{ in } \Sigma_{\lambda_0}
		\end{align*}
		in the weak sense,
		we can exploit the strong comparison principle to deduce that
		\begin{equation}\label{pos}
			u < u_{\lambda_0} \quad\text{ in } \Sigma_{\lambda_0}\setminus Z_u,
		\end{equation}
		where
		\[
		Z_u=\{x\in\mathbb{R}^N\setminus\{0\} \mid |\nabla u(x)|=0\}.
		\]		
		Indeed, by the strong comparison principle (see \cite[Proposition 1.5.2]{MR3676369} or \cite[Theorem 2.5.2]{MR2356201}), we have that either $u = u_{\lambda_0}$ in $\mathcal{U}$ or $u < u_{\lambda_0}$ in $\mathcal{U}$ for any connected component $\mathcal{U}$ of $\Sigma_{\lambda_0}\setminus Z_u$. If the former case happens, then for each $x \in \mathcal{U}\setminus\{0_{\lambda_0}\}$, we have
		\begin{align*}
			0 = -\Delta_p u - \frac{\mu}{|x|^p} u^{p-1} - \frac{u^{{p^*_s}-1}}{|x|^s}
			&= -\Delta_p u_{\lambda_0} - \frac{\mu}{|x|^p} u_{\lambda_0}^{p-1} - \frac{u_{\lambda_0}^{{p^*_s}-1}}{|x|^s}\\
			&> -\Delta_p u_{\lambda_0} - \frac{\mu}{|x_{\lambda_0}|^p} u_{\lambda_0}^{p-1} - \frac{u_{\lambda_0}^{{p^*_s}-1}}{|x_{\lambda_0}|^s} = 0
		\end{align*}
		in the classical sense, a contradiction (recalling that $u\in C^2(\mathcal{U})$ by standard regularity). Hence the latter case occurs and \eqref{pos} must hold.
		
		As in Step 1, there exists $\overline{R}_0>0$ such that
		\begin{equation}\label{ball_ineq2}
			u<u_\lambda \text{ in } B_{\overline{R}_0}(0_\lambda)\quad \text{ for all } \lambda \in [\lambda_0, \lambda_0/2].
		\end{equation}
		
		By Theorem \ref{th:gradient_asymptotic}, for $\overline{R}>R_1'$ we have
		\[
		Z_u \subset B_{\overline{R}}(0).
		\]
		On the other hand, from \cite[Corollary 1.1]{MR2447484} or the estimate \eqref{inverse_int} of Theorem \ref{th:hessian}, we have $|Z_u|=0$. Hence for any $\delta>0$, there exists a neighborhood $Z_u^\delta$ of $Z_u$ such that $Z_u^\delta\subset B_{\overline{R}}(0)$ and $|Z_u^\delta|\le\delta$.
		
		For $\lambda\in(\lambda_0,0)$, we set
		\[
		B_{\overline{R}}^\lambda := \Sigma_\lambda\setminus B_{\overline{R}}(0),\quad K_\delta:=\overline{(\Sigma_{\lambda_0-\delta}\cap B_{\overline{R}}(0))\setminus Z_u^\delta}
		\]
		and
		\[
		S_\delta^\lambda:=\Sigma_\lambda\cap B_{\overline{R}}(0)\setminus K_\delta = (\Sigma_\lambda\cap B_{\overline{R}}(0)\setminus \Sigma_{\lambda_0-\delta}) \cup (\Sigma_{\lambda_0-\delta}\cap Z_u^\delta).
		\]
		It is clear that
		\[
		\Sigma_\lambda=B_{\overline{R}}^\lambda\cup S_\delta^\lambda\cup K_\delta.
		\]
		
		By \eqref{pos}, we have $u<u_{\lambda_0}$ in $K_\delta$. Therefore, since $K_\delta$ is compact, there exists $\overline{\varepsilon}\in(0,|\lambda_0|/2)$ such that
		\[
		u < u_\lambda \quad \text{ in } K_\delta
		\]
		for any $\lambda \in [\lambda_0, \lambda_0+\overline{\varepsilon}]$.
		
		From now on, for $R>\overline{R}$, $\lambda \in [\lambda_0, \lambda_0+\overline{\varepsilon}]$ and $\alpha>\max\{2, p\}$, we consider functions $\eta$, $\varphi_{1,\lambda}$, $\varphi_{2,\lambda}$, $\psi_\lambda$ and $\varphi_\lambda$ defined as in Step 1. Reasoning as in Step 1, one yields to
		\begin{equation}\label{t1}
			\begin{aligned}
				&c_1 \int_{\hat B_{2R}\cap\{u\ge u_\lambda\}} \eta^\alpha u^2 (|\nabla u| + |\nabla u_\lambda|)^{p-2} |\nabla\ln u - \nabla\ln u_\lambda|^2\\
				&\le \int_{\hat B_{2R}} \left(\frac{u^{p^*_s-p}}{|x|^s} - \frac{u_\lambda^{p^*_s-p}}{|x_\lambda|^s}\right) \eta^\alpha \psi_\lambda + \frac{C}{R^\beta} + \frac{C}{R^\frac{\beta}{p}}\\
				&= \underbrace{\int_{\hat B_{2R} \cap B_{\overline{R}}^\lambda} \left(\frac{u^{p^*_s-p}}{|x|^s} - \frac{u_\lambda^{p^*_s-p}}{|x_\lambda|^s}\right) \eta^\alpha \psi_\lambda}_{J_1} + \underbrace{\int_{\hat B_{2R}\cap S_\delta^\lambda} \left(\frac{u^{p^*_s-p}}{|x|^s} - \frac{u_\lambda^{p^*_s-p}}{|x_\lambda|^s}\right) \eta^\alpha \psi_\lambda}_{J_2}\\&\qquad + \frac{C}{R^\beta} + \frac{C}{R^\frac{\beta}{p}}.
			\end{aligned}
		\end{equation}
		Here we have used the fact that $\frac{u_\lambda}{u} \ge c$ for every $\lambda \in [\lambda_0, \lambda_0+\overline{\varepsilon}]$.
		
		In order to estimate the first term on the right-hand side of \eqref{t1} we argue exactly as to estimate $I_4$ where here $\overline{R}$ plays the role of $\lambda$ in Step 1. Hence we get
		\begin{equation}\label{J1}
			J_1 \le \frac{C}{\overline{R}^{\beta^*}} \int_{\hat B_{2R} \cap B_{\overline{R}}^\lambda\cap\{u\ge u_\lambda\}} \eta^\alpha u^2 (|\nabla u| + |\nabla u_\lambda|)^{p-2} |\nabla\ln u - \nabla\ln u_\lambda|^2 + \frac{C}{R^\beta}.
		\end{equation}
		
		For the second term on the right-hand side of \eqref{t1} we reason as in Step 1 to get
		\begin{align*}
			J_2 &\le C \int_{\hat B_{2R}\cap S_\delta^\lambda} \frac{u^{p^*_s-2}}{|x|^s} \eta^\alpha \varphi_\lambda^2 \le C \int_{\hat B_{2R}\cap S_\delta^\lambda\cap\{u\ge u_\lambda\}} \frac{u^{p^*_s}}{|x|^s} \eta^\alpha (\ln u - \ln u_\lambda)^2\\
			&\le C_u \int_{\hat B_{2R}\cap S_\delta^\lambda\cap\{u\ge u_\lambda\}} (\ln u - \ln u_\lambda)^2,
		\end{align*}
		where $C_u = \frac{C}{(|\lambda_0|/2)^s}\sup_{S_\delta^{\lambda_0/2}} u^{p^*_s}$. To further estimate, we need to divide two cases by the value of $p$.
		
		In the case $p >2$, we will exploit the weighted Poincar\'e type inequality. For $x\in S_\delta^\lambda$, we set $w(x):=(\ln u-\ln u_\lambda)^+(x)$. Notice that $w=0$ on $\partial\Sigma_\lambda$. By defining $w(x+te_1)=0$ for all $x\in \partial\Sigma_\lambda\cap B_R(0)$ and $t>0$, where $e_1=(1,0,\dots,0)$, we can write
		\[
		w(x) = -\int_{0}^{+\infty} \frac{\partial w}{\partial x_1} (x+te_1) dt = -\int_{0}^{\lambda-x_1} \frac{\partial w}{\partial x_1} (x+te_1) dt \quad\text{ for all } x\in S_\delta^\lambda.
		\]
		By taking the integral in $x_2,\dots,x_N$ and using the finite covering lemma, we deduce
		\[
		|w(x)| \le C \int_{S^\lambda_\delta} \frac{|\nabla w(y)|}{|x-y|^{N-1}} dy
		\]
		for some $C>0$ and all $x \in S^\lambda_\delta$.		
		Therefore, $w$ verifies \eqref{wp_ine_cond2} with $\Omega=S^\lambda_\delta$. By Theorem \ref{th:hessian}, we see that $\rho=|\nabla u|^{p-2}$ satisfies \eqref{wp_ine_cond1}.	Hence we can apply Theorem \ref{th:wp_ine} to obtain
		\begin{equation}\label{r1.2}
			\begin{aligned}
				J_2	&\le C_u \int_{\hat B_{2R}\cap S_\delta^\lambda\cap\{u\ge u_\lambda\}} (\ln u - \ln u_\lambda)^2\\
				&\le C_P(S^\lambda_\delta) C_u \int_{\hat B_{2R}\cap S_\delta^\lambda\cap\{u\ge u_\lambda\}} |\nabla u|^{p-2}|\nabla\ln u - \nabla\ln u_\lambda|^2\\
				&\le \frac{C_P(S^\lambda_\delta) C_u}{\inf_{S^{\lambda_0/2}_\delta} u^2} \int_{\hat B_{2R}\cap S_\delta^\lambda\cap\{u\ge u_\lambda\}}  u^2 (|\nabla u| + |\nabla u_\lambda|)^{p-2} |\nabla\ln u - \nabla\ln u_\lambda|^2,
			\end{aligned}
		\end{equation}
		where $C_P(S^\lambda_\delta)\to0$ if $|S^\lambda_\delta|\to0$.
		
		Otherwise, if $1 <p <2$ one can apply the classical Poincaré inequality in order to deduce
		\begin{equation}\label{r1.2'}
			\begin{aligned}
				J_2	&\le C_u \int_{\hat B_{2R}\cap S_\delta^\lambda\cap\{u\ge u_\lambda\}} (\ln u - \ln u_\lambda)^2\\
				&\le C_P(S^\lambda_\delta) C_u \int_{\hat B_{2R}\cap S_\delta^\lambda\cap\{u\ge u_\lambda\}} |\nabla\ln u - \nabla\ln u_\lambda|^2\\
				&\le \frac{C C_P(S^\lambda_\delta) C_u}{\inf_{S^{\lambda_0/2}_\delta} u^2} \int_{\hat B_{2R}\cap S_\delta^\lambda\cap\{u\ge u_\lambda\}}  u^2 (|\nabla u| + |\nabla u_\lambda|)^{p-2} |\nabla\ln u - \nabla\ln u_\lambda|^2,
			\end{aligned}
		\end{equation}
		which can be deduced since in $\Sigma_\lambda \setminus B_{\overline{R}_0}(0_\lambda)$ one has
		\[
		\left(|\nabla u| + |\nabla u_\lambda|\right)^{2-p} \le C,
		\]
		for some constant $C$ which does not depend on $\lambda \in [\lambda_0, \lambda_0/2]$.
		
		Hence in both cases, by collecting \eqref{t1}, \eqref{J1}, \eqref{r1.2} and \eqref{r1.2'}, one has
		\begin{equation}\label{r2}
			\begin{aligned}
				&c_1 \int_{\hat B_{2R}\cap\{u\ge u_\lambda\}} \eta^\alpha u^2 (|\nabla u| + |\nabla u_\lambda|)^{p-2} |\nabla\ln u - \nabla\ln u_\lambda|^2\\
				&\le \frac{C}{\overline{R}^{\beta^*}} \int_{\hat B_{2R} \cap B_{\overline{R}}^\lambda\cap\{u\ge u_\lambda\}} \eta^\alpha u^2 (|\nabla u| + |\nabla u_\lambda|)^{p-2} |\nabla\ln u - \nabla\ln u_\lambda|^2\\&\qquad+ \frac{C C_P(S^\lambda_\delta) C_u}{\inf_{S^{\lambda_0/2}_\delta} u^2} \int_{\hat B_{2R}\cap S_\delta^\lambda\cap\{u\ge u_\lambda\}}  u^2 (|\nabla u| + |\nabla u_\lambda|)^{p-2} |\nabla\ln u - \nabla\ln u_\lambda|^2\\&\qquad + \frac{C}{R^\beta} + \frac{C}{R^\frac{\beta}{p}}.
			\end{aligned}
		\end{equation}
		
		Now we choose the parameters $\overline{R}, \delta,\overline{\varepsilon}$. First, we fix some $\theta<2^{-N}$ and fix $\overline{R}$ large so that
		\[
		\frac{C}{c_1 \overline{R}^{\beta^*}} < \theta.
		\]
		Then recalling that $|S^\lambda_\delta|\to0$ if $\delta\to0$ and $\lambda\to\lambda_0$, we can choose $\delta$ and $\overline{\varepsilon}$ small such that
		\[
		\frac{C C_P(S^\lambda_\delta) C_u}{c_1\inf_{S^{\lambda_0/2}_\delta} u^2} < \theta
		\]
		for any $\lambda \in [\lambda_0, \lambda_0+\overline{\varepsilon}]$. With this choice of parameters, from \eqref{r2}, we get
		\[
		L(R) \le \theta L(2R) + \frac{C}{R^\beta} + \frac{C}{R^\frac{\beta}{p}},
		\]
		where
		\[
		L(R) := \int_{\hat B_R\cap\{u\ge u_\lambda\}} u^2 (|\nabla u| + |\nabla u_\lambda|)^{p-2} |\nabla\ln u - \nabla\ln u_\lambda|^2.
		\]
		Using Theorems \ref{th:asymptotic} and \ref{th:gradient_asymptotic}, it is easy to check that $L(R) \le C R^N$ for $R>\overline{R}$.
		By Lemma 2.1 in \cite{MR2886112}, we deduce
		\[
		L(R) \equiv 0.
		\]
		This fact and \eqref{ball_ineq2} imply $u\le u_\lambda$ in $\Sigma_\lambda$ for any $\lambda \in [\lambda_0, \lambda_0+\overline{\varepsilon}]$. This is a contradiction with the definition of $\lambda_0$.
		
		Hence $\lambda_0=0$ and Step 2 is proved.
		
		\textit{Step 3.} Conclusion.
		
		By Step 2, we have $u(x) \le u(-x_1,x_2,\dots,x_N)$ in $\{x_1<0\}$. Since we can perform the above procedure in the opposite direction, we have $u(x) \ge u(-x_1,x_2,\dots,x_N)$ in $\{x_1<0\}$. Hence $u$ is symmetric about the plane $\{x_1=0\}$. 
		
		Furthermore, by the moving plane procedure exploited in any direction $\nu\in \mathbb{S}^{N-1}$ we finally get that $u$ is radial and radially decreasing about $0$. That is, $u=u(|x|)$ with $u'(r)\le 0$ for $r>0$.
		
		It remains to show that
		\[
		u'(r) < 0 \quad\text{ for } r>0.
		\]
		By contradiction, assume there exists $r_0>0$ such that $u'(r_0)=0$. Notice that $u>u(r_0)$ in $B_{r_0}(0)$ (otherwise, we have $|Z_u|>0$, which is a contradiction with Theorem \ref{th:hessian}). Applying H\"opf's boundary lemma (see \cite[Proposition 1.3.13]{MR3676369} or \cite{MR768629}) to the positive solution $w=u-u(r_0)$ of the problem 
		\[
		\begin{cases}
			-\Delta_p w = \frac{\mu}{|x|^p} u^{p-1} + \frac{u^{p^*_s-1}}{|x|^s} &\text{ in } B_{r_0}(0),\\
			w = 0 &\text{ on } \partial B_{r_0}(0),
		\end{cases}
		\]
		we deduce $u'(r_0)<0$. This is a contradiction with our previous assumption.
		
		This completes the proof of the theorem.
	\end{proof}
	
	\section*{Acknowledgments}
	This research is funded by Vietnam National Foundation for Science and Technology Development (NAFOSTED) under grant number 101.02-2023.35.
	
	\bibliographystyle{abbrvurl}
	\bibliography{../../../references}

\begin{thebibliography}{10}

\bibitem{MR2233146}
B.~Abdellaoui, V.~Felli, and I.~Peral.
\newblock Existence and nonexistence results for quasilinear elliptic equations
  involving the {$p$}-{L}aplacian.
\newblock {\em Boll. Unione Mat. Ital. Sez. B Artic. Ric. Mat. (8)},
  9(2):445--484, 2006.

\bibitem{MR143162}
A.~D. Alexandrov.
\newblock A characteristic property of spheres.
\newblock {\em Ann. Mat. Pura Appl. (4)}, 58:303--315, 1962.
\newblock \href {https://doi.org/10.1007/BF02413056}
  {\path{doi:10.1007/BF02413056}}.

\bibitem{MR768824}
L.~Caffarelli, R.~Kohn, and L.~Nirenberg.
\newblock First order interpolation inequalities with weights.
\newblock {\em Compositio Math.}, 53(3):259--275, 1984.
\newblock URL: \url{http://www.numdam.org/item?id=CM_1984__53_3_259_0}.

\bibitem{MR982351}
L.~A. Caffarelli, B.~Gidas, and J.~Spruck.
\newblock Asymptotic symmetry and local behavior of semilinear elliptic
  equations with critical {S}obolev growth.
\newblock {\em Comm. Pure Appl. Math.}, 42(3):271--297, 1989.
\newblock \href {https://doi.org/10.1002/cpa.3160420304}
  {\path{doi:10.1002/cpa.3160420304}}.

\bibitem{MR1121147}
W.~X. Chen and C.~Li.
\newblock Classification of solutions of some nonlinear elliptic equations.
\newblock {\em Duke Math. J.}, 63(3):615--622, 1991.
\newblock \href {https://doi.org/10.1215/S0012-7094-91-06325-8}
  {\path{doi:10.1215/S0012-7094-91-06325-8}}.

\bibitem{MR4343877}
G.~Ciraolo and R.~Corso.
\newblock Symmetry for positive critical points of
  {C}affarelli-{K}ohn-{N}irenberg inequalities.
\newblock {\em Nonlinear Anal.}, 216:112683, 23, 2022.
\newblock \href {https://doi.org/10.1016/j.na.2021.112683}
  {\path{doi:10.1016/j.na.2021.112683}}.

\bibitem{MR4135671}
G.~Ciraolo, A.~Figalli, and A.~Roncoroni.
\newblock Symmetry results for critical anisotropic {$p$}-{L}aplacian equations
  in convex cones.
\newblock {\em Geom. Funct. Anal.}, 30(3):770--803, 2020.
\newblock \href {https://doi.org/10.1007/s00039-020-00535-3}
  {\path{doi:10.1007/s00039-020-00535-3}}.

\bibitem{MR2096703}
L.~Damascelli and B.~Sciunzi.
\newblock Regularity, monotonicity and symmetry of positive solutions of
  {$m$}-{L}aplace equations.
\newblock {\em J. Differential Equations}, 206(2):483--515, 2004.
\newblock \href {https://doi.org/10.1016/j.jde.2004.05.012}
  {\path{doi:10.1016/j.jde.2004.05.012}}.

\bibitem{MR709038}
E.~DiBenedetto.
\newblock {$C\sp{1+\alpha }$} local regularity of weak solutions of degenerate
  elliptic equations.
\newblock {\em Nonlinear Anal.}, 7(8):827--850, 1983.
\newblock \href {https://doi.org/10.1016/0362-546X(83)90061-5}
  {\path{doi:10.1016/0362-546X(83)90061-5}}.

\bibitem{MR2886112}
A.~Farina, L.~Montoro, and B.~Sciunzi.
\newblock Monotonicity and one-dimensional symmetry for solutions of
  {$-\Delta_pu=f(u)$} in half-spaces.
\newblock {\em Calc. Var. Partial Differential Equations}, 43(1-2):123--145,
  2012.
\newblock \href {https://doi.org/10.1007/s00526-011-0405-z}
  {\path{doi:10.1007/s00526-011-0405-z}}.

\bibitem{MR1616905}
J.~P. Garc\'{\i}a~Azorero and I.~Peral~Alonso.
\newblock Hardy inequalities and some critical elliptic and parabolic problems.
\newblock {\em J. Differential Equations}, 144(2):441--476, 1998.
\newblock \href {https://doi.org/10.1006/jdeq.1997.3375}
  {\path{doi:10.1006/jdeq.1997.3375}}.

\bibitem{2017arXiv170304353G}
J.~{Garcia-Melian}.
\newblock {Nonexistence of positive solutions for Henon equation}.
\newblock {\em arXiv e-prints}, page arXiv:1703.04353, Mar 2017.
\newblock \href {https://arxiv.org/abs/1703.04353} {\path{arXiv:1703.04353}}.

\bibitem{MR1695021}
N.~Ghoussoub and C.~Yuan.
\newblock Multiple solutions for quasi-linear {PDE}s involving the critical
  {S}obolev and {H}ardy exponents.
\newblock {\em Trans. Amer. Math. Soc.}, 352(12):5703--5743, 2000.
\newblock \href {https://doi.org/10.1090/S0002-9947-00-02560-5}
  {\path{doi:10.1090/S0002-9947-00-02560-5}}.

\bibitem{MR964617}
M.~Guedda and L.~V\'{e}ron.
\newblock Local and global properties of solutions of quasilinear elliptic
  equations.
\newblock {\em J. Differential Equations}, 76(1):159--189, 1988.
\newblock \href {https://doi.org/10.1016/0022-0396(88)90068-X}
  {\path{doi:10.1016/0022-0396(88)90068-X}}.

\bibitem{MR3467702}
C.-J. He and C.-L. Xiang.
\newblock Uniqueness of positive radial solutions to singular critical growth
  quasilinear elliptic equations.
\newblock {\em Ann. Acad. Sci. Fenn. Math.}, 41(1):143--166, 2016.
\newblock \href {https://doi.org/10.5186/aasfm.2016.4110}
  {\path{doi:10.5186/aasfm.2016.4110}}.

\bibitem{MR2388757}
D.~Kang.
\newblock On the quasilinear elliptic problems with critical {S}obolev-{H}ardy
  exponents and {H}ardy terms.
\newblock {\em Nonlinear Anal.}, 68(7):1973--1985, 2008.
\newblock \href {https://doi.org/10.1016/j.na.2007.01.024}
  {\path{doi:10.1016/j.na.2007.01.024}}.

\bibitem{MR4611678}
P.~Le and D.~H.~T. Le.
\newblock Classification of positive solutions to {$p$}-{L}aplace equations
  with critical {H}ardy-{S}obolev exponent.
\newblock {\em Nonlinear Anal. Real World Appl.}, 74:Paper No. 103949, 13,
  2023.
\newblock \href {https://doi.org/10.1016/j.nonrwa.2023.103949}
  {\path{doi:10.1016/j.nonrwa.2023.103949}}.

\bibitem{MR969499}
G.~M. Lieberman.
\newblock Boundary regularity for solutions of degenerate elliptic equations.
\newblock {\em Nonlinear Anal.}, 12(11):1203--1219, 1988.
\newblock \href {https://doi.org/10.1016/0362-546X(88)90053-3}
  {\path{doi:10.1016/0362-546X(88)90053-3}}.

\bibitem{MR2447484}
H.~Lou.
\newblock On singular sets of local solutions to {$p$}-{L}aplace equations.
\newblock {\em Chin. Ann. Math. Ser. B}, 29(5):521--530, 2008.
\newblock \href {https://doi.org/10.1007/s11401-007-0312-y}
  {\path{doi:10.1007/s11401-007-0312-y}}.

\bibitem{MR4124427}
F.~Oliva, B.~Sciunzi, and G.~Vaira.
\newblock Radial symmetry for a quasilinear elliptic equation with a critical
  {S}obolev growth and {H}ardy potential.
\newblock {\em J. Math. Pures Appl. (9)}, 140:89--109, 2020.
\newblock \href {https://doi.org/10.1016/j.matpur.2020.06.004}
  {\path{doi:10.1016/j.matpur.2020.06.004}}.

\bibitem{MR4628991}
M.~Pu, S.~Huang, and Q.~Tian.
\newblock Asymptotic behaviors of solutions to quasilinear elliptic equation
  with {H}ardy potential and critical {S}obolev exponent.
\newblock {\em Qual. Theory Dyn. Syst.}, 22(4):Paper No. 153, 31, 2023.
\newblock \href {https://doi.org/10.1007/s12346-023-00847-3}
  {\path{doi:10.1007/s12346-023-00847-3}}.

\bibitem{MR2356201}
P.~Pucci and J.~Serrin.
\newblock {\em The maximum principle}, volume~73 of {\em Progress in Nonlinear
  Differential Equations and their Applications}.
\newblock Birkh\"{a}user Verlag, Basel, 2007.

\bibitem{MR3459013}
B.~Sciunzi.
\newblock Classification of positive
  {$\mathcal{D}^{1,p}(\mathbb{R}^N)$}-solutions to the critical {$p$}-{L}aplace
  equation in {$\mathbb{R}^N$}.
\newblock {\em Adv. Math.}, 291:12--23, 2016.
\newblock \href {https://doi.org/10.1016/j.aim.2015.12.028}
  {\path{doi:10.1016/j.aim.2015.12.028}}.

\bibitem{MR333220}
J.~Serrin.
\newblock A symmetry problem in potential theory.
\newblock {\em Arch. Rational Mech. Anal.}, 43:304--318, 1971.
\newblock \href {https://doi.org/10.1007/BF00250468}
  {\path{doi:10.1007/BF00250468}}.

\bibitem{MR4194414}
S.~Shakerian and J.~V\'{e}tois.
\newblock Sharp pointwise estimates for weighted critical {$p$}-{L}aplace
  equations.
\newblock {\em Nonlinear Anal.}, 206:112236, 18, 2021.
\newblock \href {https://doi.org/10.1016/j.na.2020.112236}
  {\path{doi:10.1016/j.na.2020.112236}}.

\bibitem{MR463908}
G.~Talenti.
\newblock Best constant in {S}obolev inequality.
\newblock {\em Ann. Mat. Pura Appl. (4)}, 110:353--372, 1976.
\newblock \href {https://doi.org/10.1007/BF02418013}
  {\path{doi:10.1007/BF02418013}}.

\bibitem{MR1364003}
S.~Terracini.
\newblock On positive entire solutions to a class of equations with a singular
  coefficient and critical exponent.
\newblock {\em Adv. Differential Equations}, 1(2):241--264, 1996.

\bibitem{MR727034}
P.~Tolksdorf.
\newblock Regularity for a more general class of quasilinear elliptic
  equations.
\newblock {\em J. Differential Equations}, 51(1):126--150, 1984.
\newblock \href {https://doi.org/10.1016/0022-0396(84)90105-0}
  {\path{doi:10.1016/0022-0396(84)90105-0}}.

\bibitem{MR768629}
J.~L. V\'{a}zquez.
\newblock A strong maximum principle for some quasilinear elliptic equations.
\newblock {\em Appl. Math. Optim.}, 12(3):191--202, 1984.
\newblock \href {https://doi.org/10.1007/BF01449041}
  {\path{doi:10.1007/BF01449041}}.

\bibitem{MR3676369}
L.~V\'{e}ron.
\newblock {\em Local and global aspects of quasilinear degenerate elliptic
  equations}.
\newblock World Scientific Publishing Co. Pte. Ltd., Hackensack, NJ, 2017.
\newblock Quasilinear elliptic singular problems.
\newblock \href {https://doi.org/10.1142/9850} {\path{doi:10.1142/9850}}.

\bibitem{MR3411668}
J.~V\'{e}tois.
\newblock A priori estimates and application to the symmetry of solutions for
  critical {$p$}-{L}aplace equations.
\newblock {\em J. Differential Equations}, 260(1):149--161, 2016.
\newblock \href {https://doi.org/10.1016/j.jde.2015.08.041}
  {\path{doi:10.1016/j.jde.2015.08.041}}.

\bibitem{MR3581526}
C.~Xiang.
\newblock Gradient estimates for solutions to quasilinear elliptic equations
  with critical {S}obolev growth and hardy potential.
\newblock {\em Acta Math. Sci. Ser. B (Engl. Ed.)}, 37(1):58--68, 2017.
\newblock \href {https://doi.org/10.1016/S0252-9602(16)30115-1}
  {\path{doi:10.1016/S0252-9602(16)30115-1}}.

\bibitem{MR3369267}
C.-L. Xiang.
\newblock Asymptotic behaviors of solutions to quasilinear elliptic equations
  with critical {S}obolev growth and {H}ardy potential.
\newblock {\em J. Differential Equations}, 259(8):3929--3954, 2015.
\newblock \href {https://doi.org/10.1016/j.jde.2015.05.007}
  {\path{doi:10.1016/j.jde.2015.05.007}}.

\end{thebibliography}
	
\end{document}